\documentclass[swedish,english]{article}
\usepackage[T1]{fontenc}
\usepackage[latin9]{inputenc}
\usepackage{textcomp}
\usepackage{amsmath}
\usepackage{amsthm}
\usepackage{amssymb}
\usepackage{graphicx}
\usepackage{esint}

\makeatletter
\theoremstyle{plain}
\newtheorem{thm}{\protect\theoremname}[section]
  \theoremstyle{plain}
  \newtheorem{lem}[thm]{\protect\lemmaname}

\usepackage{hyperref}
\usepackage{nameref}
\usepackage{mathtools}
\mathtoolsset{showonlyrefs} % Only shows equation labels for labels that are referred to.

\DeclareMathOperator{\dist }{dist}

\makeatletter
\let\orgdescriptionlabel\descriptionlabel
\renewcommand*{\descriptionlabel}[1]{%
  \let\orglabel\label
  \let\label\@gobble
  \phantomsection
  \edef\@currentlabel{#1}%
  \let\label\orglabel
  \orgdescriptionlabel{#1}%
}
\makeatother

\makeatother

\usepackage{babel}
  \addto\captionsenglish{\renewcommand{\lemmaname}{Lemma}}
  \addto\captionsenglish{\renewcommand{\theoremname}{Theorem}}
  \addto\captionsswedish{\renewcommand{\lemmaname}{Lemma}}
  \addto\captionsswedish{\renewcommand{\theoremname}{Teorem}}
  \providecommand{\lemmaname}{Lemma}
\providecommand{\theoremname}{Theorem}

\begin{document}

\title{Classification of Blow-ups and Free Boundaries of Solutions to Unstable
Free Boundary Problems}

\author{Andreas Minne}

\maketitle
\selectlanguage{swedish}%
\global\long\def\dist{\operatorname{dist}}

\selectlanguage{english}%
\global\long\def\df{\mathrel{\mathop:}=}

\global\long\def\fd{=\mathrel{\mathop:}}

\begin{abstract}
In general, solutions $u$ to
\[
\Delta u(\mathbf{x})=f(\mathbf{x})\chi_{\{u>\psi\}}
\]
are not $C^{1,1}$, even for $f$ smooth and $\psi(\mathbf{x})\equiv0$.
Points around which $u$ is not $C^{1,1}$ are called \emph{singular
points}, and the set of all such points, the \emph{singular set}.
In this article we analyze blow-ups, the free boundary $\partial\{u>\psi\}$,
and the singular set close to singular points $\mathbf{x}^{0}=(x^{0},y^{0},z^{0})$
in $\mathbb{R}^{3}$. We show that blow-ups of the form 
\[
\lim_{j\to\infty}\frac{u(r_{j}\cdot+\mathbf{x}^{0})}{\|u\|_{L^{\infty}(B_{r_{j}}(\mathbf{x}^{0}))}},
\]
$r_{j}\to0^{+}$ are unique, the free boundary $\partial\{u>\psi\}$
is up to rotations close to the surfaces $(x-x^{0})^{2}+(y-y^{0})^{2}=2(z-z^{0})^{2}$
or $(x-x^{0})^{2}=(z-z^{0})^{2}$, and that singular points are either
isolated or contained in a $C^{1}$ curve. The methods of the proofs
are based on projecting the solutions $u$ on the space of harmonic
two-homogeneous polynomials.
\end{abstract}

\section{Introduction and Main Results}

The obstacle problem
\begin{equation}
\Delta u=\chi_{\{u>0\}},\label{eq:obstacle-problem}
\end{equation}
$u\ge0$, has been completely solved for quite some time (here $\chi_{\Omega}$
denotes the indicator function which is equal to one on $\Omega$
and zero outside) in the sense that questions concerning existence,
uniqueness, stability and regularity of solutions as well as the free
boundary have been well answered \cite{MR0318650,MR0454350}. However,
its evil twin

\begin{equation}
\Delta u=-\chi_{\{u>0\}},\label{eq:unstable-obstacle-problem}
\end{equation}
also called the \emph{unstable obstacle problem} has not until recently
received some attention. The author can see at least two reasons for
this: one is that it is quite new from a modelling perspective, and
another reason is that techniques that can handle the obstacle problem
turn out to fail for the unstable obstacle problem. Typical models
that, after simplification such as looking at the stationary case,
yield \eqref{eq:unstable-obstacle-problem} are those of solid combustion
and composite membranes. We refer to \cite{MR2286633} and the references
therein for a more complete description of the origin of \eqref{eq:unstable-obstacle-problem}.

One qualitative difference between \eqref{eq:obstacle-problem} and
\eqref{eq:unstable-obstacle-problem} is uniqueness of solutions,
and can already be illustrated in one dimension by noting that $u=\frac{x}{2}(1-x)$
and $u=0$ both solve $u''=-\chi_{\{u>0\}}$ in $(0,1)$ with $u(0)=u(1)=0$.
Another important distinction is the difference in regularity: it
is well-known in the free boundary community that even though the
right hand side in \eqref{eq:obstacle-problem} is only bounded, for
which $C^{1,\alpha}$ holds in general (and not better) where $0<\alpha<1$,
one can show that solutions are $C^{1,1}$. On the other hand it has
been proven in \cite{MR2289547,ASW2} that there exist solutions in
two dimensions and higher to \eqref{eq:unstable-obstacle-problem}
that have unbounded $C^{1,1}$ norm. Note however that the term ``obstacle''
in unstable obstacle problem is somewhat misleading since the solution
is allowed to both lie below and above the ``obstacle'' function
(which is zero) unlike the classical obstacle problem where $u\ge0$
is imposed.

Perhaps more interesting than regularity of solutions is the regularity
of the free boundary $\partial\{u>0\}$. For the obstacle problem,
the free boundary is analytic around ``thick'' points $\mathbf{x}^{0}$,
where blow-ups, i.e., limits 
\[
v(\mathbf{x})\df\lim_{r\to\infty}\frac{u(r\mathbf{x}+\mathbf{x}^{0})}{r^{2}}
\]
are halfspace solutions of the form $\frac{1}{2}(\mathbf{x}\cdot\mathbf{e}_{\mathbf{x}^{0}})_{+}^{2}$
for directions $\mathbf{e}_{\mathbf{x}^{0}}\in\partial B_{1}$, while
it is not so easy to classify around ``thin'' a.k.a. singular points
where blow-ups are two-homogeneous polynomials. In two dimensions
it can be a cusp \cite{MR0516201}, and in general the set of singular
points can be included in a union of $C^{1}$ manifolds.

For the unstable obstacle problem, the singular points are analysed
in \cite{ASW1,ASW2}. We would also like to warn the reader that the
term ``singular'' for \eqref{eq:obstacle-problem} means points
$\mathbf{x}^{0}$ for which blow-ups are two-homogeneous polynomials
while for \eqref{eq:unstable-obstacle-problem} it is used for points
$\mathbf{x}^{0}$ such that $u\not\in C^{1,1}(B_{r}(\mathbf{x}^{0}))$
for any $r>0$. Because of this lack of regularity, limits of $\frac{u(r\mathbf{x}+\mathbf{x}^{0})}{\|u\|_{L^{\infty}(B_{r}(\mathbf{x}^{0}))}}$
are considered (since $|\frac{u(r\mathbf{x}+\mathbf{x}^{0})}{r^{2}}|$
diverges as $r\to0^{+}$) and can via the Weiss monotonicity formula
be shown to be a harmonic two-homogeneous polynomial \cite{MR2286633}.

This article will answer whether the results in \cite{ASW2} can be
extended when perturbing the right hand side both with respect to
the coefficient in front of the indicator function, and the level
set. In other words we answer what the blow-up limits at singular
points (as used for the unstable obstacle problem), the singular set,
and free boundary $\partial\{u>\psi\}$ of solutions $u$ to 
\begin{equation}
\Delta u(\mathbf{x})=f(\mathbf{x})\chi_{\{u>\psi\}}\label{eq:main}
\end{equation}
look like locally, given that $f$ is Dini-continuous, $\psi\in C^{1,\alpha}(B_{1})$,
and $|\frac{\psi(r\mathbf{x}+\mathbf{x}^{0})}{r^{2}}|$ is bounded
uniformly with respect to $r$ for singular points $\mathbf{x}^{0}$.
More precisely, we show that given a bounded solution $u$ to \eqref{eq:main}
in $B_{1}\subset\mathbb{R}^{3}$,
\begin{description}
\item [{(i)}] blow-ups of $u$ of the form $\lim_{j\to\infty}\frac{u(r_{j}\mathbf{x}+\mathbf{x}^{0})}{\|u\|_{L^{\infty}(B_{r_{j}}(\mathbf{x}^{0}))}}$
at singular points $\mathbf{x}^{0}=(x^{0},y^{0},z^{0})$ are unique
and, up to a rotation, of the form
\begin{gather*}
\pm\left(\frac{(x-x^{0})^{2}+(y-y^{0})^{2}}{2}-(z-z^{0})^{2}\right),\\
\text{ or }(x-x^{0})^{2}-(z-z^{0})^{2},
\end{gather*}
cf. Theorem \ref{thm:symmetric-blowups}, \ref{thm:blowups-unique-1}
and \ref{thm:unstable-case}.
\item [{(ii)}] if the blow-up at a singular point ${\bf x}^{0}$ is of
the form $\frac{(x-x^{0})^{2}+(y-y^{0})^{2}}{2}-(z-z^{0})^{2}$ or
$-\left(\frac{(x-x^{0})^{2}+(y-x^{0})^{2}}{2}-(z-z^{0})^{2}\right)$,
then the free boundary $\partial\{u>\psi\}$ looks locally like the
cone
\[
\frac{(x-x^{0})^{2}+(y-y^{0})^{2}}{2}=(z-z^{0})^{2}
\]
up to a $C^{1}$ perturbation, cf. Theorem \ref{thm:cone-blowups}.
Furthermore, such a singular point is isolated, cf. Theorem \ref{thm:structure-singularset}.
\item [{(iii)}] if the blow-up at a singular point ${\bf x}^{0}$ is of
the form $(x-x^{0})^{2}-(z-z^{0})^{2}$, then the set of singular
points is locally contained in a $C^{1}$-curve, cf. Theorem \ref{thm:structure-singularset}.
Furthermore, $\{u=\psi\}\cap K_{c}$ is for any $c>0$ locally contained
in two $C^{1}$-manifolds, intersecting orthogonally, where $K_{c}=\{y^{2}<c(x^{2}+z^{2})\}$,
cf. Theorem \ref{thm:unstable-case}.
\end{description}
\begin{figure}
\centering{}\includegraphics[width=1\linewidth]{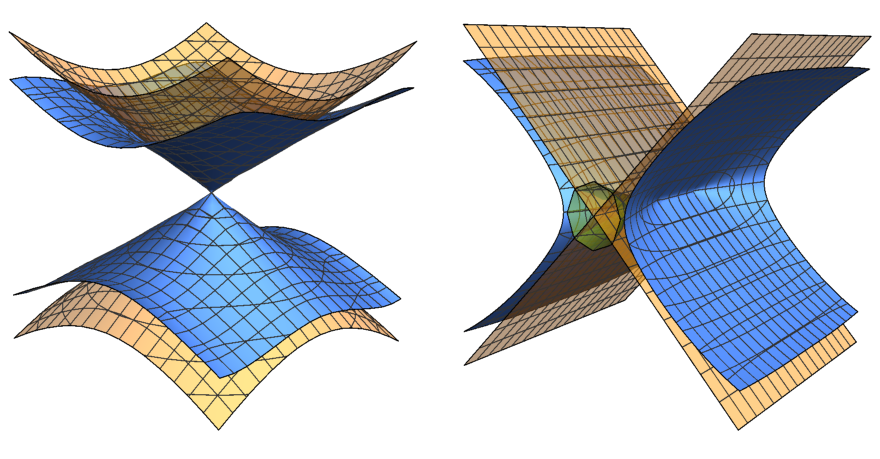}\caption{Illustration of (ii) and (iii); The behaviour inside the cusp is unclear. }
\end{figure}

Note that even though the main results are given above, these follow
by the same proof methods as in \cite{ASW2} with some key lemmas
replaced. These lemmas are presented and proved in §\ref{sec:Growth-of-Solutions}
and can be considered the substantial new contributions by the author.
The problem \eqref{eq:main} is an example of a class of problems
of the type $\Delta u=f(x,u)$ where $f$ is discontinuous in $u$.

Finally we remark that the analysis mostly takes place in $\mathbb{R}^{3}$,
but many of the results would be true in higher dimensions as well
(following \cite{MR3034585} instead of \cite{ASW2}).

The rest of the paper has the following structure: §\ref{sec:Setup-and-Notation}
introduces the notation, and also the assumptions that nonetheless
will be repeated in the lemmas, theorems and corollaries. §\ref{sec:Background-Preliminaries}
presents some important auxiliary results not proven by the author,
while the main lemmas are proven in §\ref{sec:Growth-of-Solutions}
from which the results, presented in §\ref{sec:Extension-of-Results}
follows.

\paragraph{Acknowledgements}

The author sincerely thanks Henrik Shahgholian for introducing the
problem, and John Andersson for patient guidance through some of the
hard technicalities in \cite{ASW2}.

\section{Setup and Notation\label{sec:Setup-and-Notation}}

As mentioned in the introduction we consider \eqref{eq:main}, 
\begin{align*}
\Delta u(\mathbf{x}) & =f(\mathbf{x})\chi_{\{u>\psi\}}
\end{align*}
in the ball $B_{1}\subset\mathbb{R}^{n}$ centered around the origin.
Here, $f$ will be assumed to be Dini continuous with modulus of continuity
$\omega_{f}$. The function $\chi_{\Omega}$ is the characteristic
function of the set $\Omega$, 
\[
\chi_{\Omega}\df\begin{cases}
1, & \mathbf{x}\in\Omega,\\
0 & \mathbf{x}\in\Omega^{c}.
\end{cases}
\]
The set $\{u>\psi\}$ is shorthand for the set $\{\mathbf{x}\in B_{1}\,:\,u(\mathbf{x})>\psi(\mathbf{x})\}$.
The analysis often occurs in $\mathbb{R}^{3}$, for which points are
denoted by $\mathbf{x}=(x,y,z)$.

It will also be convenient to introduce the notation
\begin{equation}
h_{r,\mathbf{x}^{0}}(\mathbf{x})\df\frac{h(r\mathbf{x}+\mathbf{x}^{0})}{r^{2}},\label{eq:blowupdefinition}
\end{equation}
and it will be assumed throughout that $|\psi_{r}|\le C$ for some
universal constant $C>0$, the reason for which will be apparent in
the proof of Lemma \ref{lem:Pi-growth-at-singular-points-nonconstant}
and Lemma \ref{lem:ur-Pi-decay}. A \emph{singular point} $\mathbf{x}^{0}$
of a solution $u$ of \eqref{eq:main} is a point such that, for any
$r>0$, $u\not\in C^{1,1}(B_{r}(\mathbf{x}^{0}))$, and the set of
all such points, the \emph{singular set}, is denoted by $S^{u}$.
Because of the translation invariance of the Laplacian, we often fix
$\mathbf{x}^{0}=0$, and let $u_{r}\df u_{r,0}$ to decrease the notational
burden. A consequence of this definition is that functions, other
than polynomials, depending on a parameter differently than in \eqref{eq:blowupdefinition}
will have the dependence written as a superscript rather than a subscript.

The analysis will be done in the interior, hence no boundary conditions
will be given. The solution will however be assumed globally bounded,
say by $M>0$. An important tool will be the following operator $\Pi:W^{2,2}(B_{1})\times(0,s)\times B_{1-s}\to\mathcal{\mathbb{HP}}_{2}$,
where $\mathcal{\mathbb{HP}}_{2}$ is the space of homogeneous polynomials
of degree two, 
\begin{align}
 & \frac{1}{|B_{r}(\mathbf{x}^{0})|}\int_{B_{r}(\mathbf{x}^{0})}\Big|\frac{D^{2}u(r\mathbf{y})}{r^{2}}-D^{2}\Pi(u,r,\mathbf{x}^{0})\Big|^{2}\,d\mathbf{y}\nonumber \\
 & =\inf_{p\in\mathbb{HP}_{2}}\frac{1}{|B_{r}(\mathbf{x}^{0})|}\int_{B_{r}(\mathbf{x}^{0})}\Big|\frac{D^{2}u(r\mathbf{y})}{r^{2}}-D^{2}p\Big|^{2}\,d\mathbf{y}.\label{eq:Pi-def}
\end{align}
Here $|B_{r}(\mathbf{x}^{0})|$ denotes the $n$-dimensional Lebesgue
measure of $B_{r}(\mathbf{x}^{0})$, $D^{2}u$ the Hessian of $u$,
and $W^{2,2}(B_{1})$ the usual Sobolev space. Again to ease the notation,
$\Pi(u,r)\df\Pi(u,r,0)$. Some properties of $\Pi$ can be found in
§\ref{sec:Background-Preliminaries}.

An auxiliary problem of interest is

\begin{equation}
\begin{cases}
\Delta Z_{p}=-\chi_{\{p>0\}} & \text{in }\mathbb{R}^{n},\\
Z_{p}(0)=\nabla Z_{p}(0)=0,\\
\lim_{|\mathbf{x}|\to\infty}\frac{Z_{p}(\mathbf{x})}{|\mathbf{x}|^{3}}=\Pi(Z_{p},1)=0.
\end{cases}\label{eq:blow-up-problem}
\end{equation}
where the function $p$ lies in $\mathbb{HP}_{2}$.

The coordinate system is in the analysis often rotated; we therefore
introduce $\mathcal{R}$ to be the set of all rotations in $\mathbb{R}^{3}$.

\section{Preliminaries\label{sec:Background-Preliminaries}}

The original problem 
\[
\Delta u=-\chi_{\{u>0\}}
\]
 was first introduced to the free boundary community by Monneu and
Weiss \cite{MR2286633}. There it was correctly conjectured that there
exists singular points for solutions $u$, i.e., points $\mathbf{x}^{0}$
such that $u$ is not $C^{1,1}$ in a neighbourhood of $\mathbf{x}^{0}$.
This was later proved by Andersson and Weiss \cite{MR2289547}, and
a small cascade of papers followed by Andersson, Shahgholian and Weiss
\cite{ASW1,ASW2,MR3034585} where the structure of the singular set
is analyzed in detail.

Many of the tools used in these papers apply equally well to our problem
\[
\Delta u=f\chi_{\{u>\psi\}},
\]
and are presented below, including references to the proofs. First
we will recall some properties of $\Pi$ introduced in the previous
section, and can be found in \cite[Lemma 3.6]{ASW2}:
\begin{lem}
\label{lem:Projection-properties}Let $\Pi(u,r,\mathbf{x}^{0})$ be
the two-homogeneous polynomial given by \eqref{eq:Pi-def}. Then
\begin{description}
\item [{(i)}] $\Pi$ is a projection operator, i.e., $\Pi(\Pi(u,r),r)=\Pi(u,r)$.
\item [{(ii)}] $\Pi(h,r)=\Pi(h,s)$ for any harmonic function function
$h$ and $0<r$, $s<1$. 
\end{description}
\end{lem}
It is also useful to note that, with the notation introduced in §\ref{sec:Setup-and-Notation},
$\Pi(u,r)=\Pi(u_{r},1)$ by the definition of $\Pi$.

The second tool is a result concerning Fourier expansions of solutions
to the Poisson equation $\Delta u=\sigma$ when the right hand side
is zero-homogeneous, i.e. $\sigma(r\mathbf{x})=\sigma(\mathbf{x})$
for all $r>0$.
\begin{thm}
[\cite{MR1444257}, Theorem 3.1] Let $u$ solve $\Delta u=\sigma$
for a zero-homogeneous function $\sigma$ such that \textup{where
$p_{j}$ are $j$-homogeneous harmonic polynomials, and assume }$u(0)=|\nabla u(0)|=\lim_{|x|\to\infty}\frac{u(\mathbf{x})}{|\mathbf{x}|^{3}}=0$\textup{.
Then 
\[
u(\mathbf{x})=\frac{a_{2}}{n+2}p_{2}\ln|\mathbf{x}|+|\mathbf{x}|^{2}\sum_{j\ne2}\frac{a_{j}}{(n+j)(j-2)}p_{j}(\mathbf{x}/|\mathbf{x}|).
\]
}
\[
\sigma(\mathbf{x})=\sum_{j=0}^{\infty}a_{j}p_{j}(\mathbf{x}/|\mathbf{x}|),
\]
where $p_{j}$ are $j$-homogeneous harmonic polynomials, and assume
$u(0)=|\nabla u(0)|=\lim_{|x|\to\infty}\frac{u(\mathbf{x})}{|\mathbf{x}|^{3}}=0$.
Then 
\[
u(\mathbf{x})=\frac{a_{2}}{n+2}p_{2}\ln|\mathbf{x}|+|\mathbf{x}|^{2}\sum_{j\ne2}\frac{a_{j}}{(n+j)(j-2)}p_{j}(\mathbf{x}/|\mathbf{x}|).
\]
\end{thm}
In \cite{ASW2} the authors apply this result to the case $\sigma=-\chi_{\{p>0\}}$,
$p\in\mathbb{HP}_{2}$ so that
\[
Z_{p}(\mathbf{x})=\frac{a_{2}}{n+2}p_{2}\ln|\mathbf{x}|+|\mathbf{x}|^{2}\sum_{j\ne2}\frac{a_{j}}{(n+j)(j-2)}p_{j}(\mathbf{x}/|\mathbf{x}|).
\]
For reasons related to the analysis, and that will be apparent in
§\ref{sec:Extension-of-Results}, we introduce a parametrization of
harmonic polynomials in diagonal form with supremum norm one, $p_{\delta}\df\pm[(1/2+\delta)x^{2}+(1/2-\delta)y^{2}-z^{2}]$,
for $\delta\in[0,1/2]$. Since $\delta$ is unique, we can define
$\delta(u_{r})$ as the $\delta$ such that $\frac{\Pi(u_{r},1)}{\|\Pi(u_{r},1)\|_{L^{\infty}(B_{1})}}$
after rotations and relabeling of the coordiante axes are of the form
$p_{\delta}$. Choosing the basis $3x^{2}-|x|^{2}$, $3x^{2}-|y|^{2}$
and $3x^{2}-|z|^{2}$ for the axisymmetric second order polynomials,
it follows that
\begin{align}
 & \Pi(Z_{p_{\delta}},1/2)\nonumber \\
 & =\frac{(3A_{x}(\delta)-A(\delta))x^{2}+(3A_{y}(\delta)-A(\delta))z^{2}+(3A_{z}(\delta)-A(\delta))z^{2}}{\|3x^{2}-1\|_{L^{2}(\partial B_{1})}},\label{eq:Pi(Z,1/2)}
\end{align}
where the coefficients $A_{x}$, $A_{y}$, $A_{z}$ and $A$ are calculated
by
\begin{align*}
A_{x}(\delta) & =(-\chi_{\{p_{\delta}>0\}},x^{2})\df\int_{\partial B_{1}}-\chi_{\{p_{\delta}>0\}}x^{2}\,dS,\\
A_{y}(\delta) & =(-\chi_{\{p_{\delta}>0\}},y^{2})\df\int_{\partial B_{1}}-\chi_{\{p_{\delta}>0\}}y^{2}\,dS,\\
A_{z}(\delta) & =(-\chi_{\{p_{\delta}>0\}},z^{2})\df\int_{\partial B_{1}}-\chi_{\{p_{\delta}>0\}}z^{2}\,dS,\\
A(\delta) & =(-\chi_{\{p_{\delta}>0\}},1)\df\int_{\partial B_{1}}-\chi_{\{p_{\delta}>0\}}\,dS.
\end{align*}
What follows is a collection of properties involving the coefficients
of $\Pi(Z_{p_{\delta}},1/2)$ and an estimate of the growth of $|Cp_{\delta}+\Pi(Z_{p_{\delta}},1/2)|$
from below.
\begin{lem}
[\cite{ASW2}, Lemma 4.5-6]\label{lem:Coefficient-estimates}For $\delta\in(0,1/2)$,
$\eta_{0}$ and large enough $C$,
\begin{align*}
\sup_{B_{1}}|Cp_{\delta}+\Pi(Z_{p_{\delta}},1/2)| & \ge C+\eta_{0}.\\
(1+2\delta)\frac{3A_{y}(\delta)-A(\delta)}{3A_{x}(\delta)-A(\delta)}-1+2\delta & \le4\delta\\
(1+2\delta)\frac{3A_{y}(\delta)-A(\delta)}{3A_{x}(\delta)-A(\delta)}-1+2\delta & \ge2\delta,\qquad\delta\le c_{0}\\
(1+2\delta)\frac{3A_{y}(\delta)-A(\delta)}{3A_{x}(\delta)-A(\delta)}-1+2\delta & \ge c(\beta),\qquad\delta\in[\beta,1-\beta]
\end{align*}

\end{lem}
The following useful lemma -- found in \cite{MR1814358} -- is of
independent interest and gives an estimate of the size of sublevel
sets of $|p|$ where $p$ is a second order polyomial.
\begin{lem}
[\cite{MR1814358}, Corollary 4.1]\label{lem:polynomial-set-size}Let
$p$ be a second order polynomial in $\mathbb{R}^{n}$ With $\|p\|_{L^{\infty}(Q_{1})}\le1$,
$Q_{1}$ being the unit cube. Then
\[
|\{|p|\le\epsilon\}\cap Q_{1}|\le C(n,\alpha)\epsilon^{\alpha},\qquad\alpha\in(0,1/4).
\]

\end{lem}
From the estimates given below we have uniform bounds when considering
blow-ups. Heuristically they state that if $u_{r}$ behaves badly
at some point, then that behaviour is inherited by $\Pi(u_{r},1$).
\begin{lem}
[\cite{ASW1}, Lemma 5.1]\label{lem:projection-C1alpha} Let $u$
solve $|\Delta u|\le C$ in $B_{1}$ such that $\|u\|_{L^{\infty}(B_{1})}\le M$.
If $\mathbf{x}^{0}\in B_{1/2}$, $0<r\le1/4$, and $u(\mathbf{x}^{0})=|\nabla u(\mathbf{x}^{0})|=0$,
then
\begin{align}
\|u_{r,\mathbf{x}^{0}}-\Pi(u_{r},1)\|_{C^{1,\alpha}(\overline{B}_{1})} & \le C(n,M,\alpha),\label{eq:ur-Piur-C1alpha}\\
\|u_{r,\mathbf{x}^{0}}-\Pi(u_{r},1)\|_{W^{2,p}(B_{1})} & \le C(n,M,p).\label{eq:ur-Piur-W2p}
\end{align}

\end{lem}
The proof in \cite[ Lemma 5.1]{ASW1} can be applied and goes through
word by word, except that $|\Delta u|\le C$ replaces $\Delta u=-\chi_{\{u>0\}}$.

A final remark is that the problem

\[
\Delta u(x)=f(x)\chi_{\{u>0\}}+g(x)
\]
can be reduced to \eqref{eq:main} by replacing $u$ with $u+\psi$
where $\psi$ solves $\Delta\psi=-g$, and the analysis applies as
long as $\psi$ is $C^{1,1}$.

\section{Growth of Solutions\label{sec:Growth-of-Solutions}}

The first lemma in this section shows that the growth of $\Pi(u_{r},1,\mathbf{x}^{0})$
is logarithmic in $r$ around singular points $\mathbf{x}^{0}$. The
analysis is performed at the origin, but is identical for any singular
point $\mathbf{x}^{0}$.
\begin{lem}
\label{lem:Pi-growth-at-singular-points-nonconstant}Let $u$ solve
\eqref{eq:main} in $B_{1}\subset\mathbb{R}^{3}$ for $\|u\|_{L^{\infty}(B_{1})}\le M$
such that $u(0)=|\nabla u(0)|=0$, $\sup_{0<r\le1/4}|\psi_{r}|\le C_{\psi}$,
and $f$ is continuous with modulus of continuity $\omega_{f}$ and
$f(0)=-a$, $a>0$. Then there exist constants $K_{0}$ and $r_{0}=r_{0}(K_{0},C_{\psi},f)$
such that if
\[
\|\Pi(u_{r},1)\|_{L^{\infty}(B_{1})}\ge K_{0},\qquad0<r\le r_{0}
\]
then
\begin{align}
\|\Pi(u_{r},2^{-j})\|_{L^{\infty}(B_{1})} & \ge\|\Pi(u_{r},1)\|_{L^{\infty}(B_{1})}+\frac{ja\eta_{0}}{2}\label{eq:Pi-growth-below}\\
\|\Pi(u_{r},2^{-j})\|_{L^{\infty}(B_{1})} & \le\|\Pi(u_{r},1)\|_{L^{\infty}(B_{1})}+ja\kappa_{0}\label{eq:Pi-growth-above}
\end{align}
and
\begin{align}
\|u\|_{L^{\infty}(B_{s})} & \ge\frac{1}{16}\left(\left(\frac{s}{r}\right)^{2}\|u\|_{L^{\infty}(B_{r})}+\eta_{0}as^{2}\ln(r/s)\right),\label{eq:u-growth}\\
\|u\|_{L^{\infty}(B_{s})} & \le16\left(\left(\frac{s}{r}\right)^{2}\|u\|_{L^{\infty}(B_{r})}+\kappa_{0}as^{2}\ln(r/s)\right),
\end{align}
$0<s\le r$. Here $\eta_{0}$ is the constant in Lemma \ref{lem:Coefficient-estimates}
and $\kappa_{0}=\kappa_{0}(M,\alpha)$.\end{lem}
\begin{proof}
Assume for the moment that $f(0)=-1$. The proof goes via contradiction:
take a sequence $\{u^{k}\}$ of solutions to \eqref{eq:main} bounded
by $M$ such that $\|\Pi(u_{r_{k}}^{k},1)\|_{L^{\infty}(B_{1})}\ge k$
while 
\begin{equation}
\|\Pi(u_{r_{k}}^{k},1/2)\|_{L^{\infty}(B_{1})}\le\|\Pi(u_{r_{k}}^{k},1)\|_{L^{\infty}(B_{1})}+\eta_{0}/2.\label{eq:Pi-small-growth}
\end{equation}
By Lemma \ref{lem:projection-C1alpha} $\{u_{r_{k}}^{k}-\Pi(u_{r_{k}}^{k},1)\}$
converges in $C_{\text{loc}}^{1,\alpha}(\mathbb{R}^{n})$, up to a
subsequence, to a function $v$ as $r_{k}\to0^{+}$. Passing again
to a subsequence such that $p\df\lim_{k}\frac{\Pi(u_{r_{k}}^{k},1)}{\|\Pi(u_{r_{k}}^{k},1)\|_{L^{\infty}(B_{1})}}\in\mathbb{HP}_{2}$
and $p_{k}\df\frac{\Pi(u_{r_{k}}^{k},1)}{\|\Pi(u_{r_{k}}^{k},1)\|_{L^{\infty}(B_{1})}}$,
we see that $v$ solves
\begin{equation}
\Delta v=\lim_{k\to\infty}f(r_{k}\mathbf{x})\chi_{\{u_{r_{k}}^{k}>\psi_{r_{k}}\}}=-\chi_{\{p>0\}}.\label{eq:limit-equation}
\end{equation}
Indeed,
\[
u_{r_{k}}^{k}>\psi_{r_{k}}\Leftrightarrow\frac{u_{r_{k}}^{k}-\Pi(u_{r_{k}}^{k},1)}{\|\Pi(u_{r_{k}}^{k},1)\|_{L^{\infty}(B_{1})}}+\frac{\Pi(u_{r_{k}}^{k},1)}{\|\Pi(u_{r_{k}}^{k},1)\|_{L^{\infty}(B_{1})}}>\frac{\psi_{r_{k}}}{\|\Pi(u_{r_{k}}^{k},1)\|_{L^{\infty}(B_{1})}}
\]
and from the assumptions, 
\[
\max\left\{ \frac{|\psi_{r_{k}}|}{\|\Pi(u_{r_{k}}^{k},1)\|_{L^{\infty}(B_{1})}},\frac{|u_{r_{k}}^{k}-\Pi(u_{r_{k}}^{k},1)|}{\|\Pi(u_{r_{k}}^{k},1)\|_{L^{\infty}(B_{1})}}\right\} \le\frac{C}{k}\to0,\qquad k\to\infty,
\]
from which $u_{r_{k}}^{k}>\psi_{r_{k}}$ is equivalent to $p>h^{k}$
for some function $h^{k}$ tending to zero as $k$ increases. Therefore,
for any $\epsilon>0$ and $\phi\in C_{c}^{1}(B_{R})$, 
\begin{align*}
 & \lim_{k\to\infty}\Big|\int_{B_{R}}\phi(f(r_{k}x)\chi_{\{h^{k}>0\}}+\chi_{\{p>0\}})\,dx)\Big|\\
 & \le\|\phi\|_{L^{\infty}(B_{R})}\lim_{k\to\infty}\int_{B_{R}}|f(r_{k}x)\chi_{\{u_{r_{k}}^{k}>\psi_{r_{k}}\}}-f(0)\chi_{\{u_{r_{k}}^{k}>\psi_{r_{k}}\}}|\,dx\\
 & \qquad+\|\phi\|_{L^{\infty}(B_{R})}\lim_{k\to\infty}\int_{B_{R}}|\chi_{\{p>0\}}-\chi_{\{p>h^{k}\}}|\,dx\\
 & \le\|\phi\|_{L^{\infty}(B_{R})}(|B_{R}|\lim_{k\to\infty}\omega_{f}(r_{k})+|\{|p|\le\epsilon\}\cap B_{R}|)\\
 & \le C(R)\|\phi\|_{L^{\infty}(B_{R})}\epsilon^{\alpha},
\end{align*}
by Lemma \ref{lem:polynomial-set-size}. Hence \eqref{eq:limit-equation}
is proven and $v=Z_{p}$ by uniqueness for $Z_{p}$ given in §\ref{sec:Setup-and-Notation}.
Consequently
\[
\lim_{j\to\infty}[\Pi(u_{r_{k}}^{k},1/2)-\Pi(u_{r_{k}}^{k},1)]=\Pi(Z_{p},1/2).
\]
However, then by Lemma \eqref{lem:Coefficient-estimates}, 
\begin{align*}
\|\Pi(u_{r_{k}}^{k},1/2)\|_{L^{\infty}(B_{1})} & \ge\|\Pi(u_{r_{k}}^{k},1)+\Pi(Z_{p_{k}},1/2)\|_{L^{\infty}(B_{1})}+o(1)\\
 & \ge\|\Pi(u_{r_{k}}^{k},1)\|_{L^{\infty}(B_{1})}+\eta_{0}+o(1)
\end{align*}
which is a contradiction to \eqref{eq:Pi-small-growth}, so \eqref{eq:Pi-growth-below}
holds for $j=1$ and $a=1$. This can be iterated, recalling that
$\Pi(u_{r},s)=\Pi(u_{sr},1)$,

\begin{align*}
 & \|\Pi(u_{r},2^{-j})\|_{L^{\infty}(B_{1})}\\
 & =\|\Pi(u_{2^{-j+1}r},1/2)\|_{L^{\infty}(B_{1})}\ge\|\Pi(u_{2^{-j+1}r},1)\|_{L^{\infty}(B_{1})}+\eta_{0}/2\\
 & =\|\Pi(u_{2^{-j+2}r},1/2)\|_{L^{\infty}(B_{1})}+\eta_{0}/2\ge\|\Pi(u_{2^{-j+2}r},1)\|_{L^{\infty}(B_{1})}+2\eta_{0}/2\\
 & \ge\ldots\ge\|\Pi(u_{r},1)\|_{L^{\infty}(B_{1})}+j\eta_{0}/2,
\end{align*}
which is the first inequality in \eqref{eq:Pi-growth-below}.

The bound from above follows from the mean value property of harmonic
functions,
\begin{align*}
 & |\Pi(u_{r},1/2)-\Pi(u_{r},1)|^{2}\\
 & =\Big|\fint_{B_{r/2}}\Pi(u_{r},1/2)-\Pi(u_{r},1)\,d\mathbf{x}\Big|^{2}\\
 & \le2\fint_{B_{r/2}}|\Pi(u_{r},1/2)-D^{2}u_{r}|^{2}\,d\mathbf{x}+2\fint_{B_{r/2}}|D^{2}u_{r}-\Pi(u_{r},1)|^{2}\,d\mathbf{x}\\
 & \le2C+2^{n+1}\fint_{B_{r}}|D^{2}u-\Pi(u_{r},1)|^{2}\,d\mathbf{x}\le2^{n+2}C
\end{align*}
for $C$ as in \eqref{eq:ur-Piur-W2p}. Therefore
\[
\|\Pi(u_{r},1/2)\|_{L^{\infty}(B_{1})}\le\|\Pi(u_{r},1)\|_{L^{\infty}(B_{1})}+\kappa_{0}
\]
for $\kappa_{0}=2^{n+2}C$, which again can be iterated,
\begin{align*}
 & \|\Pi(u_{r},2^{-j})\|_{L^{\infty}(B_{1})}\\
 & =\|\Pi(u_{2^{-j+1}r},1/2)\|_{L^{\infty}(B_{1})}\le\|\Pi(u_{2^{-j+1}r},1)\|_{L^{\infty}(B_{1})}+\kappa_{0}\\
 & =\|\Pi(u_{2^{-j+2}r},1/2)\|_{L^{\infty}(B_{1})}+\kappa_{0}\le\|\Pi(u_{2^{-j+2}r},1)\|_{L^{\infty}(B_{1})}+2\kappa_{0}\\
 & \le\ldots\le\|\Pi(u_{r},1)\|_{L^{\infty}(B_{1})}+j\kappa_{0}.
\end{align*}

By Lemma \ref{lem:projection-C1alpha} with $K_{0}>2C$ as in \eqref{eq:ur-Piur-C1alpha},
\begin{equation}
\frac{1}{2}\|\Pi(u_{r},1)\|_{L^{\infty}(B_{1})}\le\|u_{r}\|_{L^{\infty}(B_{1})}\le2\|\Pi(u_{r},1)\|_{L^{\infty}(B_{1})}.\label{eq:u-and-Pi-comparable}
\end{equation}
Let $2^{-j-1}r\le s\le2^{-j}r$, then
\begin{align*}
\frac{\|u\|_{L^{\infty}(B_{s})}}{s^{2}} & \ge\frac{\|u\|_{L^{\infty}(B_{2^{-j-1}r})}}{(2^{-j}r)^{2}}=\frac{1}{4}\|u_{2^{-j-1}r}\|_{L^{\infty}(B_{1})}\\
 & \ge\{\eqref{eq:u-and-Pi-comparable}\}\ge\frac{1}{8}\|\Pi(u_{r},2^{-j-1})\|_{L^{\infty}(B_{1})}\\
 & \ge\{\eqref{eq:Pi-growth-below}\}\ge\frac{1}{8}(\|\Pi(u_{r},1)\|_{L^{\infty}(B_{1})}+(j+1)\eta_{0}/2)\\
 & \ge\{\eqref{eq:u-and-Pi-comparable}\}\ge\frac{1}{16}\|u_{r}\|_{L^{\infty}(B_{1})}+\frac{1}{16}(j+1)\eta_{0}\\
 & \ge\frac{1}{16}(\|u_{r}\|_{L^{\infty}(B_{1})}+j\eta_{0}).
\end{align*}
and
\begin{align*}
\frac{\|u\|_{L^{\infty}(B_{s})}}{s^{2}} & \le4\frac{\|u\|_{L^{\infty}(B_{2^{-j}r})}}{(2^{-j}r)^{2}}=4\|u_{2^{-j}r}\|_{L^{\infty}(B_{1})}\\
 & \le8\|\Pi(u_{r},2^{-j})\|_{L^{\infty}(B_{1})}\le8(\|\Pi(u_{r},1)\|_{L^{\infty}(B_{1})}+j\kappa_{0})\\
 & \le16\|u_{r}\|_{L^{\infty}(B_{1})}+8j\kappa_{0}\\
 & \le16(\|u_{r}\|_{L^{\infty}(B_{1})}+j\kappa_{0}).
\end{align*}

where $\ln(\frac{r}{s})-1\le\ln(2)^{-1}\ln(\frac{r}{s})-1\le j\le\ln(2)^{-1}\ln(\frac{r}{s})\le2\ln(\frac{r}{s})$
has been used. Finally, apply the proof to $u/a$ if $a\ne1$ to get
\eqref{eq:Pi-growth-below} and \eqref{eq:Pi-growth-above}.
\end{proof}
The following lemma improves the estimate in Lemma \ref{lem:projection-C1alpha},
and says, combined with the previous lemma, that $|u_{r}-\Pi(u_{r},1)|$
decays like $(|\ln r|)^{-\alpha}$ at singular points.
\begin{lem}
\label{lem:ur-Pi-decay}Let $u$ solve \eqref{eq:main} in $B_{1}\subset\mathbb{R}^{n}$
for $\|u\|_{L^{\infty}(B_{1})}\le M$ such that $u(0)=|\nabla u(0)|=0$,
$\sup_{0<r\le1/4}|\psi_{r}|\le C_{\psi}$, and $f$ is Dini continuous
with modulus of continuity $\omega_{f}$ and $f(0)=-a$, $a>0$. Then
if $g^{r}$ solves
\begin{align*}
\Delta g^{r} & =f(0)\chi_{\{\Pi(u_{r},1)>0\}}-f(r\mathbf{x})\chi_{\{u_{r}>\psi_{r}\}}\qquad\text{in }B_{1},\\
g^{r} & =0\qquad\text{on }\partial B_{1},
\end{align*}
for $0<r<c(\omega_{f})$,
\begin{align}
\|D^{2}g^{r}\|_{L^{2}(B_{1})} & \le C(M,n,\alpha,f,\psi)(\|\Pi(u_{r},1)\|_{L^{\infty}(B_{1})})^{-\alpha},\label{eq:L2diffest-1}\\
\|\Pi(g^{r},s)\|_{L^{\infty}(B_{1})} & \le C(M,n,\alpha,f,\psi,c)(\|\Pi(u_{r},1)\|_{L^{\infty}(B_{1})})^{-\alpha},\label{eq:L2est->pointwiseest-1}
\end{align}
for $\text{\ensuremath{s\ge c>0}.}$\end{lem}
\begin{proof}
From Lemma \ref{lem:projection-C1alpha} and the boundedness of $\psi_{r}$,
\[
|u_{r}-\psi_{r}-\Pi(u_{r},1)|\le|u_{r}-\Pi(u_{r},1)|+|\psi_{r}|\le C(n,M,\psi,f)
\]
so $\Pi(u_{r},1)<-C$ implies that $u_{r}-\psi_{r}<0$, and $\Pi(u_{r},1)\ge C$
similarly gives that $u_{r}\ge\psi_{r}$. Therefore $|\Delta g^{r}|\le|f(r\mathbf{x})-f(0)|\le\omega_{f}(r)$
outside the set $\{|\Pi(u_{r},1)|\le C\}$ in $B_{1}$. Now Lemma
\ref{lem:polynomial-set-size} concludes that
\begin{align*}
\|\Delta g^{r}\|_{L^{2}(B_{1})}^{2} & \le|\{|\Pi(u_{r},1)|\le C\}|+C\omega_{f}(r)^{2}\\
 & =\bigg|\Big\{\frac{|\Pi(u_{r},1)|}{\|\Pi(u_{r},1)\|_{L^{\infty}(B_{1})}}\le\frac{C}{\sup_{B_{1}}|\Pi(u_{r},1)|}\Big\}\bigg|+C\omega_{f}(r)^{2}\\
 & \le C\sup_{B_{1}}|\Pi(u_{r},1)|^{-\alpha}+C\omega_{f}(r)^{2}\\
 & \le C(\|\Pi(u_{r},1)\|_{L^{\infty}(B_{1})})^{-\alpha},\qquad\alpha\in(0,1/4),\,r\le c(\omega_{f}),
\end{align*}
where the last inequality is due to $\omega_{f}$ being Dini continuous
while $(\|\Pi(u_{r},1)\|_{L^{\infty}(B_{1})})^{-\alpha}\propto|\ln r|^{-\alpha}$
by Lemma \ref{lem:Pi-growth-at-singular-points-nonconstant} which
is non-Dini. Finally $L^{p}$-theory (see e.g. Theorem 9.11 in \cite{gilbarg:01})
implies \eqref{eq:L2diffest-1}.

Now let $\Pi(g^{r},s)=\sum_{i,j=1}^{n}a_{ij}x_{i}x_{j}$. Then
\begin{align*}
\|D^{2}\Pi(g^{r},s)\|_{L^{2}(B_{1})}^{2} & =\int_{B_{1}}|D^{2}\Pi(g^{r},s)(\mathbf{x})|^{2}\,d\mathbf{x}\le\int_{B_{1}}|D^{2}\frac{g^{r}(s\mathbf{x})}{s^{2}}|^{2}\,dx\\
 & =\frac{1}{s^{n}}\int_{B_{s}}|D^{2}g^{r}(\mathbf{y})|^{2}\,d\mathbf{y}\le C\|D^{2}g^{r}\|_{L^{2}(B_{1})}^{2}\\
 & \le\{\eqref{eq:L2diffest-1}\}\le C(\|\Pi(u_{r},1)\|_{L^{\infty}(B_{1})})^{-\alpha}
\end{align*}
and 
\[
\|D^{2}\Pi(g^{r},s)\|_{L^{2}(B_{1})}^{2}\ge\int_{B_{1}}\sum_{i,j=1}^{n}|a_{ij}|\,dx=C\sum_{i,j=1}^{n}|a_{ij}|
\]
so \eqref{eq:L2est->pointwiseest-1} holds.
\end{proof}
From this result, we can control how $\Pi(u,r)$ changes as $r$ is
halved. The proof goes through as in Corollary 7.3 in \cite{ASW2}
with \cite[Lemma 7.2]{ASW2} replaced by Lemma \ref{lem:ur-Pi-decay};
the details are included for completeness.
\begin{lem}
\label{lem:r-halving-lemma}Let $u$ solve \eqref{eq:main} in $B_{1}\subset\mathbb{R}^{n}$
for $\|u\|_{L^{\infty}(B_{1})}\le M$ such that $u(0)=|\nabla u(0)|=0$,
$\sup_{0<r\le1/4}|\psi_{r}|\le C_{\psi}$, and $f$ is Dini continuous
with modulus of continuity $\omega_{f}$ and $f(0)=-a$, $a>0$. Then
\begin{align*}
 & \|\Pi(u_{r},1/2)-\Pi(u_{r},1)-\Pi(Z_{\Pi(u_{r},1)},1/2)\|_{L^{\infty}(B_{1})}\\
 & \le C(M,n,\alpha,\psi,f)(\|\Pi(u_{r},1)\|_{L^{\infty}(B_{1})})^{-\alpha}.
\end{align*}
\end{lem}
\begin{proof}
Write
\[
u_{r}=\Pi(u_{r},1)+Z_{\Pi(u_{r},1)}+\tilde{g}^{r}+\tilde{h}^{r},
\]
where $\tilde{g}^{r}$ and $\tilde{h}^{r}$ have the properties
\begin{align*}
\Delta\tilde{g}^{r} & =\Delta(u_{r}-Z_{\Pi(u_{r},1)}),\qquad\text{in }B_{1}\\
\Delta\tilde{h}^{r} & =0,\qquad\text{in }B_{1}
\end{align*}
and $\tilde{g}^{r}(0)=|\nabla\tilde{g}^{r}(0)|=\Pi(\tilde{g}^{r},1)=\tilde{h}^{r}(0)=|\nabla\tilde{h}^{r}(0)|=\Pi(\tilde{h}^{r},1)=0$.
Next, let $g^{r}$ solve
\begin{align*}
\Delta g^{r} & =\Delta\tilde{g}^{r}\qquad\text{in }B_{1},\\
g^{r} & =0\qquad\text{on }\partial B_{1}.
\end{align*}
Then $h^{r}\df\tilde{g}^{r}+\tilde{h}^{r}-g^{r}$ is harmonic. By
Lemma \ref{lem:ur-Pi-decay},
\[
\|\Pi(g^{r},1/2)\|_{L^{\infty}(B_{1})}\le C(M,n,\alpha,\psi,f)(\|\Pi(u_{r},1)\|_{L^{\infty}(B_{1})})^{-\alpha},
\]
and
\begin{align*}
\|\Pi(h^{r},1/2)\|_{L^{\infty}(B_{1})} & =\|\Pi(h^{r},1)\|_{L^{\infty}(B_{1})}=\|\Pi(g^{r},1)\|_{L^{\infty}(B_{1})}\\
 & \le C(M,n,\alpha,\psi,f)(\|\Pi(u_{r},1)\|_{L^{\infty}(B_{1})})^{-\alpha}
\end{align*}
since $h^{r}$ is harmonic and by the assumptions $\Pi(\tilde{g}^{r},1)=\Pi(\tilde{h}^{r},1)=0$.
Therefore
\begin{align*}
\Pi(u_{r},1/2) & =\Pi(\Pi(u_{r},1),1/2)+\Pi(Z_{\Pi(u_{r},1)},1/2)+\Pi(g^{r},1/2)+\Pi(h^{r},1/2)\\
 & =\Pi(u_{r},1)+\Pi(Z_{\Pi(u_{r},1)},1/2)+\Pi(g^{r},1/2)+\Pi(h^{r},1/2)
\end{align*}
by Lemma \ref{lem:Projection-properties}. We conclude,
\begin{align*}
 & \|\Pi(u_{r},1/2)-\Pi(u_{r},1)-\Pi(Z_{\Pi(u_{r},1)},1/2)\|_{L^{\infty}(B_{1})}\\
 & \le\|\Pi(g^{r},1/2\|_{L^{\infty}(B_{1})}+\|\Pi(h^{r},1/2\|_{L^{\infty}(B_{1})}\\
 & \le C(M,n,\alpha,\psi,f)(\|\Pi(u_{r},1)\|_{L^{\infty}(B_{1})})^{-\alpha}.
\end{align*}

\end{proof}

\section{Results\label{sec:Extension-of-Results}}

As mentioned in the introduction, the following theorems are the fruits
of harvest from the preparation made in the earlier sections. The
proofs follow the ones in \cite{ASW2} with the following replacements:
Lemma \ref{lem:projection-C1alpha}, Lemma \ref{lem:Pi-growth-at-singular-points-nonconstant},
Lemma \ref{lem:r-halving-lemma}, Theorem \ref{thm:symmetric-blowups},
\ref{thm:blowups-unique-1} and \ref{thm:unstable-case} replace \cite[Proposition 3.7, Corollary 5.3, 7.3, Theorem 8.1, 9.1 and 11.1]{ASW2}
respectively. Note that \cite[Proposition 3.2, and Theorem 4.5]{ASW2}
used in the proofs are valid for our problem as well. With this being
said, the proofs of Theorem \ref{thm:symmetric-blowups}, \ref{thm:blowups-unique-1}
and \ref{thm:cone-blowups} will nonetheless be included to demonstrate
the techniques used. The main idea is to show monotonicity of the
parameter $\delta_{r}\df\delta(u_{r})$ in the parametrization 
\[
\frac{\Pi(u_{r},1)}{\|\Pi(u_{r},1)\|_{L^{\infty}(B_{1})}}=p_{\delta_{r}}=\pm[(1/2+\delta_{r})x^{2}+(1/2-\delta_{r})y^{2}-z^{2}]
\]
as we go from $r$ to $r/2$. This is done by using the fact that
\[
\Pi(u_{r},1)-\Pi(u_{r},1/2)\approx\Pi(Z_{\Pi(u_{r},1)},1/2)
\]
by Lemma \ref{lem:r-halving-lemma} and exploiting the estimates of
the coefficients of $A_{x}(\delta)$, $A_{y}(\delta)$, $A_{z}(\delta)$
and $A(\delta)$ given in Lemma \eqref{lem:Coefficient-estimates}.

The first result shows that, up to rotations, the blow-ups are contained
in a discrete set. Recall the notation $Z_{p}$ for solutions to \eqref{eq:blow-up-problem}
\begin{thm}
\label{thm:symmetric-blowups}Let $u$ solve \eqref{eq:main} in $B_{1}\subset\mathbb{R}^{3}$
for $\|u\|_{L^{\infty}(B_{1})}\le M$ such that $u(0)=|\nabla u(0)|=0$,
$\sup_{0<r\le1/4}|\psi_{r}|\le C_{\psi}$, and $f$ is Dini continuous
with modulus of continuity $\omega_{f}$ and $f(0)=-a$, $a>0$. Assume
also that $u$ is not $C^{1,1}$ at the origin, i.e., for any $K>0$,
\[
\|u_{r}\|_{L^{\infty}(B_{1})}\ge K
\]
 for $r\le r_{0}(K)$. Then each limit of
\[
\frac{u(r\mathbf{x})}{r^{2}}-\Pi(u_{r},1)
\]
is contained in
\[
\{Z_{p_{1}}(Q\mathbf{x})\,:\,Q\in\mathcal{R}\}\cup\{Z_{p_{2}}(Q\mathbf{x})\,:\,Q\in\mathcal{R}\}\cup\{Z_{p_{3}}(Qx)\,:\,Q\in\mathcal{R}\},
\]
for $p_{1}\df\frac{x^{2}+y^{2}}{2}-z^{2}$, $p_{2}\df-p_{1}$ and
$p_{3}\df x^{2}-z^{2}$ respectively.\end{thm}
\begin{proof}
Assume otherwise and let $u$ be a solution to \eqref{eq:main} such
that, after a rotation $Q\in\mathcal{R}$ and up to a sign, which
we without loss of generality assume to be positive,
\[
\lim_{j\to\infty}\frac{u_{r_{j}}}{\|u_{r_{j}}\|_{L^{\infty}(B_{1})}}=p_{\delta_{0}}
\]
 for some $\delta_{0}\in(0,1/2)$, and where the parametrisation $p_{\delta}=(1/2+\delta)x^{2}+(1/2-\delta)y^{2}-z^{2}$
is used. Define $\delta_{r}$ so that, after a rotation $Q(r)\in\mathcal{R}$,
\[
\frac{\Pi(u_{r},1)}{\|\Pi(u_{r},1)\|_{L^{\infty}(B_{1})}}=p_{\delta_{r}},
\]
and $\Pi(Z_{\Pi(u_{r},1)},1/2)$ by \eqref{eq:Pi(Z,1/2)} can be written
as
\[
\frac{(3A_{x}(\delta_{r})-A(\delta_{r}))x^{2}+(3A_{y}(\delta_{r})-A(\delta_{r}))z^{2}+(3A_{z}(\delta_{r})-A(\delta_{r}))z^{2}.}{\|3x^{2}-1\|_{L^{2}(\partial B_{1})}}
\]
We would like to write $\Pi(Z_{\Pi(u_{r},1)},1/2)$ in terms of $p_{\delta_{r}}$:
let $\kappa(\delta)$ be defined such that 
\[
1-2\delta+\kappa(\delta)=(1+2\delta)\frac{A_{y}(\delta)-A(\delta)}{A_{x}(\delta)-A(\delta)}.
\]
Then
\begin{align*}
 & \Pi(Z_{\Pi(u_{r},1)},1/2)\\
 & =-\frac{3A_{x}(\delta_{r})-A(\delta_{r})}{\|3x^{2}-1\|_{L^{2}(\partial B_{1})}(1+2\delta_{r})}((1+2\delta_{r})x^{2}+(1+2\delta_{r})\frac{A_{y}(\delta_{r})-A(\delta_{r})}{A_{x}(\delta_{r})-A(\delta_{r})}y^{2}\\
 & \qquad+(1+2\delta_{r})\frac{A_{z}(\delta_{r})-A(\delta_{r})}{A_{x}(\delta_{r})-A(\delta_{r})}z^{2})\\
 & =-\frac{3A_{x}(\delta_{r})-A(\delta_{r})}{\|3x^{2}-1\|_{L^{2}(\partial B_{1})}(1+2\delta_{r})}((1+2\delta_{r})x^{2}+(1-2\delta_{r}+\kappa(\delta_{r})y^{2}\\
 & \qquad+(1+2\delta_{r})\frac{A_{z}(\delta_{r})-A(\delta_{r})}{A_{x}(\delta_{r})-A(\delta_{r})}z^{2})\\
 & =-\frac{3A_{x}(\delta_{r})-A(\delta_{r})}{\|3x^{2}-1\|_{L^{2}(\partial B_{1})}(1+2\delta_{r})}((1+2\delta_{r})x^{2}+(1-2\delta_{r}+\kappa(\delta_{r})y^{2}\\
 & \qquad-(2+\kappa(\delta_{r}))z^{2}),\\
 & =-\frac{3A_{x}(\delta_{r})-A(\delta_{r})}{\|3x^{2}-1\|_{L^{2}(\partial B_{1})}(1+2\delta_{r})}(2p_{\delta_{r}}+\kappa(\delta_{r})(y^{2}-z^{2}))
\end{align*}
where we used harmonicity in the second last equality. From Lemma
\ref{lem:r-halving-lemma} we therefore have 
\begin{align}
 & |\Pi(u_{r},1/2)-\|\Pi(u_{r},1)\|_{L^{\infty}(B_{1})}p_{\delta_{r}}\nonumber \\
 & +\frac{3A_{x}(\delta_{r})-A(\delta_{r})}{\|3x^{2}-1\|_{L^{2}(\partial B_{1})}(1+2\delta_{r})}(2p_{\delta_{r}}+\kappa(\delta_{r})(y^{2}-z^{2}))|\nonumber \\
 & \le C(M,\alpha,\psi,f)(\|\Pi(u_{r},1)\|_{L^{\infty}(B_{1})})^{-\alpha},\label{eq:not-normalized}
\end{align}
from which $\Pi(u_{r},1/2)$ can be written as 
\begin{align*}
\left[\|\Pi(u_{r},1)\|_{L^{\infty}(B_{1})}(\frac{1}{2}+\delta_{r})-\frac{3A_{x}(\delta_{r})-A(\delta_{r})}{\|3x^{2}-1\|_{L^{2}(\partial B_{1})}(1+2\delta_{r})}(1+2\delta_{r})\right]x^{2}\\
+\left[\|\Pi(u_{r},1)\|_{L^{\infty}(B_{1})}(\frac{1}{2}-\delta_{r})-\frac{3A_{x}(\delta_{r})-A(\delta_{r})}{\|3x^{2}-1\|_{L^{2}(\partial B_{1})}(1+2\delta_{r})}(1-2\delta_{r}+\kappa(\delta_{r}))\right]y^{2}\\
-\left[\|\Pi(u_{r},1)\|_{L^{\infty}(B_{1})}-\frac{3A_{x}(\delta_{r})-A(\delta_{r})}{\|3x^{2}-1\|_{L^{2}(\partial B_{1})}(1+2\delta_{r})}(2+\kappa(\delta_{r}))\right]z^{2}\\
+O((\|\Pi(u_{r},1)\|_{L^{\infty}(B_{1})})^{-\alpha}).
\end{align*}
Note that $\Pi(u_{r},1/2)$ might have mixed terms with respect to
the coordinate system chosen such that $\Pi(u_{r},1)=\|\Pi(u_{r},1)\|_{L^{\infty}(B_{1})}p_{\delta_{r}}$.
However, the coefficients in front of the mixed terms are of order
$C(\sup_{B_{1}}|\Pi(u_{r},1)|)^{-\alpha}$ which can be seen by choosing
points such that $p_{\delta_{r}}$ and $y^{2}-z^{2}$ are zero. Therefore,
by normalizing the $z$-coefficient, $1/2+\delta_{r/2}$ is equal
to 
\begin{align*}
 & =\frac{\|\Pi(u_{r},1)\|_{L^{\infty}(B_{1})}(1/2+\delta_{r})-\frac{3A_{x}(\delta_{r})-A(\delta_{r})}{\|3x^{2}-1\|_{L^{2}(\partial B_{1})}(1+2\delta_{r})}(1+2\delta_{r})}{\|\Pi(u_{r},1)\|_{L^{\infty}(B_{1})}-\frac{3A_{x}(\delta_{r})-A(\delta_{r})}{\|3x^{2}-1\|_{L^{2}(\partial B_{1})}(1+2\delta_{r})}(2+\kappa(\delta_{r}))}\\
 & \qquad+\frac{O((\|\Pi(u_{r},1)\|_{L^{\infty}(B_{1})})^{-\alpha})}{|\ln r|}\\
 & =1/2+\frac{\|\Pi(u_{r},1)\|_{L^{\infty}(B_{1})}-2\frac{3A_{x}(\delta_{r})-A(\delta_{r})}{\|3x^{2}-1\|_{L^{2}(\partial B_{1})}(1+2\delta_{r})}}{\|\Pi(u_{r},1)\|_{L^{\infty}(B_{1})}-\frac{3A_{x}(\delta_{r})-A(\delta_{r})}{\|3x^{2}-1\|_{L^{2}(\partial B_{1})}(1+2\delta_{r})}(2+\kappa(\delta_{r}))}\delta_{r}\\
 & \qquad+\frac{\frac{3A_{x}(\delta_{r})-A(\delta_{r})}{\|3x^{2}-1\|_{L^{2}(\partial B_{1})}(1+2\delta_{r})}}{\|\Pi(u_{r},1)\|_{L^{\infty}(B_{1})}-\frac{3A_{x}(\delta_{r})-A(\delta_{r})}{\|3x^{2}-1\|_{L^{2}(\partial B_{1})}(1+2\delta_{r})}(2+\kappa(\delta_{r}))}\kappa(\delta_{r})\\
 & \qquad+\frac{O((\|\Pi(u_{r},1)\|_{L^{\infty}(B_{1})})^{-\alpha})}{|\ln r|}
\end{align*}
from which $\delta_{r/2}$ can be estimated,
\begin{align}
\delta_{r/2} & \le\delta_{r}-C\frac{\kappa(\delta_{r})}{\|\Pi(u_{r},1)\|_{L^{\infty}(B_{1})}}+C_{1}(\|\Pi(u_{r},1)\|_{L^{\infty}(B_{1})})^{-1-\alpha}\label{eq:delta-r/2-tau}\\
 & \le\delta_{r}-C\frac{\kappa(\delta_{r})}{|\ln r|}+\frac{O(\|\Pi(u_{r},1)\|_{L^{\infty}(B_{1})})^{-\alpha})}{|\ln r|}\label{eq:delta-r/2}
\end{align}

where we used $\|\Pi(u_{r},1)\|_{L^{\infty}(B_{1})}\propto|\ln r|$
by Lemma \ref{lem:ur-Pi-decay} and 
\[
-\infty<\frac{3A_{x}(\delta_{r})-A(\delta_{r})}{\|3x^{2}-1\|_{L^{2}(\partial B_{1})}(1+2\delta_{r})}\le-C<0.
\]
Now, if $\delta_{r}\in[\beta,1-\beta]$, $\kappa(\delta)\ge c(\beta)$
by Lemma \ref{lem:Coefficient-estimates} which we together with $\|\Pi(u_{r},1)\|_{L^{\infty}(B_{1})}\ge C$
apply in \eqref{eq:delta-r/2}, 
\[
\delta_{r/2}\le\delta_{r}-\frac{C}{|\ln r|},\qquad r\le r_{0}(\beta,M,f)
\]
hence
\[
\delta_{2^{-k}r}\le\delta_{r}-C\sum_{j=1}^{k-1}\frac{1}{j}\le\delta_{r}-C\ln(k-1)
\]
assuming that $\delta_{2^{-j}r}\in[\beta,1-\beta]$ (note that $\delta_{2^{-j}r}\le1-\beta$
for every $j$ if $\delta_{r}\le1-\beta$). From this estimate, we
see that $\delta_{2^{-k}r}\le c_{0}$ for $c_{0}$ in Lemma \ref{lem:Coefficient-estimates}
eventually for $k$ large depending on $\beta$, $M$, and $f$, which
implies that $\kappa(\delta_{2^{-k}r})\ge2\delta_{2^{-k}r}$. Let
us now consider two cases for a small constant $c=\min\{\alpha/2,C/2\kappa_{0}\}$
for $C$ as in \eqref{eq:delta-r/2-tau} and $\kappa_{0}$ in Lemma
\ref{lem:Pi-growth-at-singular-points-nonconstant}:
\begin{description}
\item [{(i)}] $\delta_{r}\ge(\|\Pi(u_{r},1)\|_{L^{\infty}(B_{1})})^{-c}$:
Plugging this into \eqref{eq:delta-r/2},
\[
\delta_{r/2}\le\delta_{r}-\frac{C}{|\ln r|^{1+c}}
\]
implying that $\delta_{2^{-j}r}$ is decreasing. If for some $j\ge1$,
$\delta_{2^{-j}r}<(\|\Pi(u_{r},2^{-j})\|_{L^{\infty}(B_{1})})^{-\alpha/2}$
see case two. If not, then consider the limit of $\delta_{2^{-j}r}$
as $j\to\infty$. This limit has to coincide with $\delta_{0}>0$,
but then we can iterate the argument above with $\beta$ replaced
by $\delta_{0}/2$, which in turn yields a logarithmic decay in $j$,
contradicting that $\delta_{2^{-j}r}\to\delta_{0}$.
\item [{(ii)}] $\delta_{r}<(\|\Pi(u_{r},1)\|_{L^{\infty}(B_{1})})^{-c}$:
Then $\delta_{2^{-j}r}<(\|\Pi(u_{r},2^{-j})\|_{L^{\infty}(B_{1})})^{-c}$
for all $j$,
\begin{align}
 & \delta_{r/2}\nonumber \\
 & \le\delta_{r}(1-\frac{C}{\|\Pi(u_{r},1)\|_{L^{\infty}(B_{1})}})+C_{1}(\|\Pi(u_{r},1)\|_{L^{\infty}(B_{1})})^{-1-\alpha},\\
 & \le\left(1-\frac{C-C_{1}(\|\Pi(u_{r},1)\|_{L^{\infty}(B_{1})})^{-\alpha+c}}{\|\Pi(u_{r},1)\|_{L^{\infty}(B_{1})}}\right)(\|\Pi(u_{r},1)\|_{L^{\infty}(B_{1})})^{-c},\label{eq:decreasing-delta-r}
\end{align}
where we used the assumption $\delta_{r}<(\|\Pi(u_{r},1)\|_{L^{\infty}(B_{1})})^{-c}$.
Recall that $\|\Pi(u_{r},1/2)\|_{L^{\infty}(B_{1})}\le\|\Pi(u_{r},1)\|_{L^{\infty}(B_{1})}+\kappa_{0}$
for $\kappa_{0}$ as in Lemma \ref{lem:Pi-growth-at-singular-points-nonconstant}.
This implies, by Taylor expansion, $\|\Pi(u_{r},1)\|_{L^{\infty}(B_{1})}\ge M$,
and $c$ small, that 
\begin{align*}
1-\frac{C-C_{1}(\|\Pi(u_{r},1)\|_{L^{\infty}(B_{1})})^{-\alpha+c}}{\|\Pi(u_{r},1)\|_{L^{\infty}(B_{1})}} & \le1-\frac{c\kappa}{\|\Pi(u_{r},1)\|_{L^{\infty}(B_{1})}}\\
 & \le\left(\frac{\|\Pi(u_{r},1/2)\|_{L^{\infty}(B_{1})}}{\|\Pi(u_{r},1)\|_{L^{\infty}(B_{1})}}\right)^{c}
\end{align*}
which we in turn put into \eqref{eq:decreasing-delta-r} to yield
$\delta_{r/2}\le\|\Pi(u_{r},1/2)\|_{L^{\infty}(B_{1})}^{-c}$. This
is iterated and consequently $\delta_{2^{-j}r}<(\|\Pi(u_{r},2^{-j})\|_{L^{\infty}(B_{1})})^{-c}$
(note that the assumption $\|\Pi(u_{r},2^{-j})\|_{L^{\infty}(B_{1})}\ge M$
holds due to $\|\Pi(u_{r},1/2)\|_{L^{\infty}(B_{1})}\ge\|\Pi(u_{r},1)\|_{L^{\infty}(B_{1})}+\eta_{0}/2$).
Since $\|\Pi(u_{r},1)\|_{L^{\infty}(B_{1})}\propto|\ln r|$, we deduce
that $\delta_{2^{-k}r}\to0$.
\end{description}
The conclusion is that if $\delta_{r}\in(0,1/2)$, then $\delta_{2^{-k}r}\to0$,
as $k\to\infty$, contradicting that $\delta_{r_{j}}\to\delta_{0}>0$.
\end{proof}
So far we have shown that $\delta_{r}$ converges to zero if there
is a subsequence $\delta_{r_{j}}$ tending to zero, and even though
it is insinuated by the proof above that the rotations $Q=Q(r)$ converges,
this will be shown rigorously below with a quantitative estimate of
the convergence speed. In particular, this shows uniqueness of blow-ups.
\begin{thm}
\label{thm:blowups-unique-1} Let $u$ solve \eqref{eq:main} in $B_{1}\subset\mathbb{R}^{3}$
for $\|u\|_{L^{\infty}(B_{1})}\le M$ such that $u(0)=|\nabla u(0)|=0$,
$\sup_{0<r\le1/4}|\psi_{r}|\le C_{\psi}$, and $f$ is Dini continuous
with modulus of continuity $\omega_{f}$ and $f(0)=-a$, $a>0$. Assume
also that for any $K>0$,
\[
\|u_{r}\|_{L^{\infty}(B_{1})}\ge K
\]
 for $r\le r_{0}(K)$. Then there exists constants $R(M,\psi,f)$,
$c(M,\psi,f)$ and $K(M,\psi,f)$ such that if
\[
s\in(0,R),\quad\sup_{B_{1}}|\Pi(u,s)|\ge K,\quad\delta(u_{s})\le c,
\]
then there is a rotation $Q\in\mathcal{R}$ such that
\begin{equation}
u_{r}-\Pi(u_{r},1)\to Z_{p_{1}}(Q\cdot)\label{eq:Z_1-conv}
\end{equation}
or 
\begin{equation}
u_{r}-\Pi(u_{r},1)\to Z_{p_{2}}(Q\cdot),\label{eq:Z_2-conv}
\end{equation}
and
\[
\Big\|\frac{\Pi(u_{r},1)}{\|\Pi(u_{r},1)\|_{L^{\infty}(B_{1})}}-\frac{p}{\|p\|_{L^{\infty}(B_{1})}}\Big\|_{L^{\infty}(B_{1})}\le C(M,\psi,f,\alpha)\left(K+\ln\big(\frac{s}{r}\big)\right)^{-c}
\]
for all $r\in(0,s)$ with $p=p_{1}\df(x^{2}+y^{2})/2-z^{2}$ in the
case \eqref{eq:Z_1-conv} and $p=p_{2}\df-p_{1}$ if \eqref{eq:Z_2-conv}
holds.\end{thm}
\begin{proof}
By \eqref{eq:not-normalized} and the bound $\kappa(\delta)\le4\delta$,
\begin{align*}
\|\Pi(u_{s},2^{-k-1})\|_{L^{\infty}(B_{1})} & =\|\Pi(u_{s},2^{-k})\|_{L^{\infty}(B_{1})}+\frac{6A_{x}(\delta_{2^{-k}s})-A(\delta_{2^{-k}s})}{\|3x^{2}-1\|_{L^{2}(\partial B_{1})}(1+2\delta_{2^{-k}s})}\\
 & \qquad+O(\delta_{2^{-k}s})+O(\Pi(u_{s},2^{-k})^{-\alpha}).
\end{align*}
If we also use $\tau_{2^{-k}s}$ to denote $\Pi(u_{s},2^{-k})\|_{L^{\infty}(B_{1})}=\Pi(u_{2^{-k}s},1)\|_{L^{\infty}(B_{1})}$
we infer that, again by \eqref{eq:not-normalized}, 
\begin{align*}
 & \bigg\|\frac{\Pi(u_{s},2^{-k})}{\|\Pi(u_{s},2^{-k})\|_{L^{\infty}(B_{1})}}-\frac{\Pi(u_{s},2^{-k-1})}{\|\Pi(u_{s},2^{-k-1})\|_{L^{\infty}(B_{1})}}\bigg\|_{L^{\infty}(B_{1})}\\
\le & \bigg\|\frac{\|\Pi(u_{s},2^{-k})\|_{L^{\infty}(B_{1})}p_{\delta_{r}}}{\|\Pi(u_{s},2^{-k})\|_{L^{\infty}(B_{1})}}\\
 & \qquad-\frac{\tau_{2^{-k}s}p_{\delta_{r}}+\frac{3A_{x}(\delta_{2^{-k}s})-A(\delta_{2^{-k}s})}{\|3x^{2}-1\|_{L^{2}(\partial B_{1})}(1+2\delta_{2^{-k}s})}(2p_{\delta_{2^{-k}s}}+\kappa(\delta_{2^{-k}s})(y^{2}-z^{2}))}{\tau_{2^{-k}s}+\frac{6A_{x}(\delta_{2^{-k}s})-2A(\delta_{2^{-k}s})}{\|3x^{2}-1\|_{L^{2}(\partial B_{1})}(1+2\delta_{2^{-k}s})}+O(\delta_{2^{-k}s})+O(\Pi(u_{s},2^{-k})^{-\alpha})}\bigg\|_{L^{\infty}(B_{1})}\\
 & \qquad+C(M,\alpha,f,\psi)(\|\Pi(u_{s},2^{-k-1})\|_{L^{\infty}(B_{1})})^{-1-\alpha}\\
\le & \bigg\|\frac{\tau_{2^{-k}s}(O(\delta_{2^{-k}s})+O(\tau_{2^{-k}s}^{-\alpha})}{\tau_{2^{-k}s}(\tau_{2^{-k}s}+\frac{6A_{x}(\delta_{2^{-k}s})-2A(\delta_{2^{-k}s})}{\|3x^{2}-1\|_{L^{2}(\partial B_{1})}(1+2\delta_{2^{-k}s})}+O(\delta_{2^{-k}s})+O(\Pi(u_{s},2^{-k})^{-\alpha}))}\bigg\|_{L^{\infty}(B_{1})}\\
 & \qquad+C(M,\alpha,f,\psi)(\|\Pi(u_{s},2^{-k-1})\|_{L^{\infty}(B_{1})})^{-1-\alpha}\\
\le & C\frac{\delta_{2^{-k}s}}{\|\Pi(u_{s},2^{-k})\|_{L^{\infty}(B_{1})}}+C(M,\alpha)(\|\Pi(u_{s},2^{-k-1})\|_{L^{\infty}(B_{1})})^{-1-\alpha}.
\end{align*}
Iteration of this inequality yields
\begin{align}
 & \bigg\|\frac{\Pi(u_{s},2^{-k})}{\|\Pi(u_{s},2^{-k})\|_{L^{\infty}(B_{1})}}-\frac{\Pi(u_{s},2^{-k-m})}{\|\Pi(u_{s},2^{-k-m})\|_{L^{\infty}(B_{1})}}\bigg\|_{L^{\infty}(B_{1})}\label{eq:sumconvergence}\\
 & \le C\sum_{j=k}^{k+m}\frac{\delta_{2^{-j}s}}{\|\Pi(u_{s},2^{-j})\|_{L^{\infty}(B_{1})}}+C(M,\alpha,f,\psi)\sum_{j=k}^{k+m}(\|\Pi(u_{s},2^{-j-1})\|_{L^{\infty}(B_{1})})^{-1-\alpha}.
\end{align}
From the proof of the previous theorem, we can divide into the cases
$\delta_{2^{-k}s}\ge(\|\Pi(u_{s},2^{-k})\|_{L^{\infty}(B_{1})})^{-c}$
or $\delta_{2^{-k}s}\le(\|\Pi(u_{s},2^{-k})\|_{L^{\infty}(B_{1})})^{-c}$,
and if the latter occurs for some $k$, then $\delta_{2^{-j}s}\le(\|\Pi(u_{s},2^{-j})\|_{L^{\infty}(B_{1})})^{-c}$
for all $j\ge k$, and it holds that such a $k$ can be chosen only
depending on $M$, $f$ and $\psi$: similarly as in the previous
proof,
\[
\delta_{s/2}\le\delta_{s}-C\frac{\delta_{s}}{\|\Pi(u_{s},1)\|_{L^{\infty}(B_{1})}},
\]
if $\delta_{s}\ge(\|\Pi(u_{s},1)\|_{L^{\infty}(B_{1})})^{-c}$. Now
if $\delta_{2^{-j}s}\ge(\|\Pi(u_{s},2^{-j})\|_{L^{\infty}(B_{1})})^{-c}$
for $j<k_{1}$,
\[
\delta_{2^{-k_{1}}s}\le\delta_{s}\prod_{j=0}^{k_{1}-1}\Big(1-\frac{C}{\|\Pi(u_{s},2^{-j})\|_{L^{\infty}(B_{1})}}\Big).
\]
From the Taylorexpansion of $\ln(1-x)$ and Lemma \ref{lem:Pi-growth-at-singular-points-nonconstant},
\begin{align*}
\ln\prod_{j=0}^{k_{1}-1}\Big(1-\frac{C}{\|\Pi(u_{s},2^{-j})\|_{L^{\infty}(B_{1})}}\Big) & =\sum_{j=0}^{k_{1}-1}\Big(1-\frac{C}{\|\Pi(u_{s},2^{-j})\|_{L^{\infty}(B_{1})}}\Big)\\
 & \le-\sum\frac{C}{2\|\Pi(u_{s},2^{-j})\|_{L^{\infty}(B_{1})}}\\
 & \le-\sum\frac{C}{2(\|\Pi(u_{s},1)\|_{L^{\infty}(B_{1})}+\kappa_{0}j)}\\
 & \le-C/2\ln\Big(\frac{k_{1}\kappa_{0}+\|\Pi(u_{s},1)\|_{L^{\infty}(B_{1})}}{\|\Pi(u_{s},1)\|_{L^{\infty}(B_{1})}}\Big)
\end{align*}
so $\delta_{2^{-k_{1}}s}\le\delta_{s}\Big(\frac{\|\Pi(u_{s},1)\|_{L^{\infty}(B_{1})}}{k_{1}\kappa_{0}+\|\Pi(u_{s},1)\|_{L^{\infty}(B_{1})}}\Big)^{C/2}$
and, if $2c\le C/2$,
\begin{align*}
\delta_{2^{-k_{1}}s} & \le\delta_{s}\Big(\frac{\|\Pi(u_{s},1)\|_{L^{\infty}(B_{1})}}{k_{1}\kappa_{0}+\|\Pi(u_{s},1)\|_{L^{\infty}(B_{1})}}\Big)^{C/2}\\
 & \le\delta_{s}\Big(\frac{\|\Pi(u_{s},1)\|_{L^{\infty}(B_{1})}}{k_{1}\kappa_{0}+\|\Pi(u_{s},1)\|_{L^{\infty}(B_{1})}}\Big)^{2c}\\
 & \le\delta_{s}\Big(\frac{\|\Pi(u_{s},1)\|_{L^{\infty}(B_{1})}}{\|\Pi(u_{s},2^{-k_{1}})\|_{L^{\infty}(B_{1})}}\Big)^{2c},
\end{align*}
where again Lemma \ref{lem:Pi-growth-at-singular-points-nonconstant}
is used in the last inequality.

Let $k_{1}\ge0$ be the first number such that $\delta_{2^{-j}s}\le(\|\Pi(u_{s},2^{-j})\|_{L^{\infty}(B_{1})})^{-c}$.
Then, by Lemma \ref{lem:r-halving-lemma}, 
\[
\|\Pi(u_{s},2^{-k_{1}+1})\|_{L^{\infty}(B_{1})}^{-c}\le\delta_{2^{-k_{1}+1}s}\le\delta_{s}\Big(\frac{\|\Pi(u_{s},1)\|_{L^{\infty}(B_{1})}}{\|\Pi(u_{s},2^{-k_{1}+1})\|_{L^{\infty}(B_{1})}}\Big)^{2c}
\]
from which
\[
(k_{1}-1)\frac{\eta_{0}}{2}+\|\Pi(u_{s},1)\|_{L^{\infty}(B_{1})}\le\|\Pi(u_{s},2^{-k_{1}+1})\|_{L^{\infty}(B_{1})}\le\delta_{s}^{1/c}\|\Pi(u_{s},1)\|_{L^{\infty}(B_{1})}^{2}
\]
and
\begin{align*}
k_{1} & \le\frac{2}{\eta_{0}}(\delta_{s}^{1/c}\|\Pi(u_{s},1)\|_{L^{\infty}(B_{1})}^{2}-\|\Pi(u_{s},1)\|_{L^{\infty}(B_{1})})+1\\
 & \le\frac{2}{\eta_{0}}\delta_{s}^{1/c}\|\Pi(u_{s},1)\|_{L^{\infty}(B_{1})}^{2}.
\end{align*}
Applying these results in \eqref{eq:sumconvergence} after letting
$m\to\infty$, 
\begin{align*}
 & \bigg\|\frac{\Pi(u_{s},2^{-k})}{\|\Pi(u_{s},2^{-k})\|_{L^{\infty}(B_{1})}}-p\bigg\|_{L^{\infty}(B_{1})}\\
 & \le C\delta_{s}^{1/c}\|\Pi(u_{s},1)\|_{L^{\infty}(B_{1})}^{2}\sum_{j=k}^{\infty}\|\Pi(u_{s},2^{-j})\|_{L^{\infty}(B_{1})}^{-1-2c}\\
 & \qquad+C(M,\alpha,f,\psi)\sum_{j=k}^{\infty}(\|\Pi(u_{s},2^{-j-1})\|_{L^{\infty}(B_{1})})^{-1-\alpha}\\
 & \le C\sum_{j=k}^{\infty}(\|\Pi(u_{s},2^{-k})\|_{L^{\infty}(B_{1})}+ja\eta_{0}/2)^{-1-2c}\\
 & \le C\|\Pi(u_{s},2^{-k})\|_{L^{\infty}(B_{1})}^{-2c}\\
 & \le C(\|\Pi(u_{s},1)\|_{L^{\infty}(B_{1})}+ka\eta_{0}/2)^{-2c}\\
 & \le C\Big(K(M)+\ln\big(\frac{s}{r}\big)\big)^{-2c},
\end{align*}
for $r\in(2^{-k-1}s,2^{-k}s)$.
\end{proof}
By the uniqueness of limits $\lim_{j\to\infty}\frac{u_{r_{j},\mathbf{x}^{0}}}{\|u_{r_{j},\mathbf{x}^{0}}\|_{L^{\infty}(B_{1})}}$
of the form 
\[
\pm\left(\frac{x^{2}+y^{2}}{2}-z^{2}\right),
\]
we can characterize the free boundary $\partial\{u>\psi\}$ of solutions
at singular points.
\begin{thm}
\label{thm:cone-blowups} Let $u$ solve \eqref{eq:main} in $B_{1}\subset\mathbb{R}^{3}$
for $\|u\|_{L^{\infty}(B_{1})}\le M$ such that $u(0)=|\nabla u(0)|=0$,
$\psi\in C^{1,\alpha}(B_{r})$,$\sup_{0<r\le1/4}|\psi_{r}|\le C_{\psi}$,
and $f$ is Dini continuous with modulus of continuity $\omega_{f}$
and $f(0)=-a$, $a>0$. Assume also that for any $K>0$,
\[
\|u_{r}\|_{L^{\infty}(B_{1})}\ge K
\]
 for $r\le r_{0}(K)$, and that 
\[
\lim_{r\to0}\frac{u_{r}}{\|u_{r}\|_{B_{1}}}=\pm p_{1}(Q\cdot)
\]
for $p_{1}=\frac{x^{2}+y^{2}}{2}-z^{2}$ and $Q\in\mathcal{R}$. Then
there are $r_{0}=r_{0}(M,f,\psi)$ and Lipschitz functions $g$ and
$h$ such that 
\[
\{u=\psi\}\cap B_{R}=\{(x,y,g(x,y)\}\cup\{(x,y,h(x,y)\}\cap B_{r_{0}}.
\]
Also, $\sqrt{2}g-\sqrt{x^{2}+y^{2}}$ and $\sqrt{2}h+\sqrt{x^{2}+y^{2}}$
are $C^{1}$ functions.\end{thm}
\begin{proof}
By Theorem \ref{thm:blowups-unique-1} we can without loss of generality
assume the limit to be $p_{1}:=\frac{x^{2}+x^{2}}{2}-z^{2}$. Since
$\frac{u_{r}-\psi_{r}}{\|u_{r}\|_{L^{\infty}(B_{1})}}\fd v^{r}\to p_{1}$
in $C^{1,\alpha}(\overline{B}_{1})$, and $|\nabla p|\ge1/2$ on $\overline{B}_{1}\backslash B_{1/2}$,
\[
\{v^{r}=0\}\cap(\overline{B}_{1}\backslash B_{1/2})\subseteq\{\mathbf{x}'\,:\,\dist(\mathbf{x}',x^{2}+y^{2}=2z^{2})<\sigma(r)\}\cap(\overline{B}_{1}\backslash B_{1/2}),
\]
for some modulus of continuity $\sigma$, and consequently $\frac{\partial v^{r}}{\partial z}\le-1/2$
on $\{u(r\mathbf{x})=\psi(r\mathbf{x})\}\cap(\overline{B}_{1}\backslash B_{1/2})$
for $r$ small enough depending on $\sigma$. From the implicit function
theorem and the $C^{1,\alpha}$-regularity of $u$ and $\psi$, $\{\mathbf{x}\,:\,u(r\mathbf{x})=\psi(r\mathbf{x}\}\cap(\overline{B}_{1}\backslash B_{1/2})$
can be expressed as a function $g^{r}(x,y)$ which is $C^{1,\alpha}(\overline{D}_{1}\backslash D_{1/2})$
with $C^{1,\alpha}$-norm independent of $r$, where $D_{r}=\{(x,y)\in\mathbb{R}^{2}\,:\,x^{2}+y^{2}<r^{2}\}$.
Let $g$ be the function given by gluing together such $g^{r}$, and
define $g(0,0)=0$. Then $g$ is Lipschitz in $D_{r_{0}}$ and $g\in C^{1,\alpha}(\overline{D}_{r_{0}}\backslash D_{s})$
for any $s>0$, $\frac{g(rx,ry)}{r}$ is bounded in $C^{1,\alpha}(\overline{D}_{1}\backslash D_{1/2})$
uniformly with respect to $r$ and $\frac{g(rx,ry)}{r}\to\frac{\sqrt{x^{2}+y^{2}}}{\sqrt{2}}$
in $C(\overline{D}_{1}\backslash D_{1/2})$ from which it follows
that $\frac{g(rx,ry)}{r}\to\frac{\sqrt{x^{2}+y^{2}}}{\sqrt{2}}$ in
$C^{1,\beta}(\overline{D}_{1}\backslash D_{1/2})$, $\beta<\alpha$.
Finally, for $(x,y)\to0$, we infer that 
\[
\bigg|\nabla g(x,y)-\nabla\frac{\sqrt{x^{2}+y^{2}}}{\sqrt{2}}\bigg|=\bigg|\nabla\frac{g(rx',ry')}{r}-\nabla\frac{\sqrt{x'^{2}+y'^{2}}}{\sqrt{2}}\bigg|\to0,\qquad x',y'\in\partial D_{1},
\]
as $r\to0$ and we conclude.
\end{proof}
As mentioned, similar results are proven for limits $\lim_{j\to\infty}\frac{u_{r_{j}}}{\|u_{r_{j}}\|_{L^{\infty}(B_{1})}}$
of the form $x^{2}-z^{2}$, even though the analysis is somewhat more
subtle due to the instability of the singularity \cite{MR2286633}.
\begin{thm}
[Corresponds to Theorem 11.1 in \cite{ASW2}]\label{thm:unstable-case}
Let $u$ solve \eqref{eq:main} in $B_{1}\subset\mathbb{R}^{3}$ for
$\|u\|_{L^{\infty}(B_{1})}\le M$ such that $u(0)=|\nabla u(0)|=0$,
$\psi\in C^{1,\alpha}(B_{r_{\psi}})$ for some $r_{\psi}>0$, $\sup_{0<r\le1/4}|\psi_{r}|\le C_{\psi}$,
and $f$ is Dini continuous with modulus of continuity $\omega_{f}$
and $f(0)=-a$, $a>0$. Assume also that for any $K>0$,
\[
\|u_{r}\|_{L^{\infty}(B_{1})}\ge K
\]
 for $r\le r_{0}(K)$. If 
\[
\lim_{j\to\infty}u_{r_{j}}-\Pi(u_{r_{j}},1)=Z_{p_{3}}(Q\cdot)
\]
for $p_{3}=x^{2}-z^{2}$, $Q\in\mathcal{R}$ and some sequence $\{r_{j}\}$,
then 
\[
\lim_{r\to0}u_{r}-\Pi(u_{r},1)=Z_{p_{3}}(Q\cdot).
\]
Furthermore, for any $c>0$, $\{u=\psi\}\cap K_{c}$ is contained
two $C^{1}$ manifolds intersecting orthogonally, where $K_{c}=\{Q\mathbf{x}\,:\,y^{2}<c(x^{2}+z^{2})\}$.
\end{thm}
Finally we state the structural properties of the singular set $S^{u}$,
which we practically split up as $S^{u}=S_{1}^{u}\cup S_{2}^{u}$
for
\begin{align*}
S_{1}^{u} & \df\bigg\{\mathbf{x}^{0}\in B_{1/2}\,:\,\lim_{r\to0}\frac{u_{r,\mathbf{x}^{0}}(Q\cdot)}{\|u_{r,\mathbf{x}^{0}}\|_{L^{\infty}(B_{1})}}\pm\bigg(\frac{x{}^{2}+y^{2}}{2}-z^{2}\bigg),\,Q\in\mathcal{R}\bigg\},\\
S_{2}^{u} & \df\bigg\{\mathbf{x}^{0}\in B_{1/2}\,:\,\lim_{r\to0}\frac{u_{r,\mathbf{x}^{0}}(Q\cdot)}{\|u_{r,\mathbf{x}^{0}}\|_{L^{\infty}(B_{1})}}x^{2}-z^{2},\,Q\in\mathcal{R}\bigg\}.
\end{align*}
 The results are that singular points in $S_{1}^{u}$ are isolated
while $S_{2}^{u}$ is locally contained in a $C^{1}$ curve.
\begin{thm}
\label{thm:structure-singularset} Let $u$ solve \eqref{eq:main}
in $B_{1}\subset\mathbb{R}^{3}$ for $\|u\|_{L^{\infty}(B_{1})}\le M$
such that $u(\mathbf{x}^{0})=|\nabla u(\mathbf{x}^{0})|=0$, $\psi\in C^{1,\alpha}(B_{r_{\psi}})$
for some $r_{\psi}>0$, $\sup_{0<r\le1/4}|\psi_{r}|\le C_{\psi}$,
and $f$ is Dini continuous with modulus of continuity $\omega_{f}$
and $f(\mathbf{x}^{0})=-a$, $a>0$. Then
\begin{description}
\item [{(i)}] if $x^{0}\in S_{1}^{u}$ , then $x^{0}$ is an isolated singular
point
\item [{(ii)}] if $x^{0}\in S_{2}^{u}$ , then $S_{2}^{u}\cap B_{r}(u)$
is contained in a $C^{1}$ curve.
\end{description}
\end{thm}
\bibliographystyle{amsalpha}
\bibliography{References}

\end{document}